\documentclass[12pt]{amsart}
\usepackage{amssymb, amscd, amsmath, amsthm, latexsym, enumerate}

\renewcommand{\geq}{\geqslant}
\renewcommand{\leq}{\leqslant}

\newtheorem{theorem}{Theorem}

\newtheorem{cor}[theorem]{Corollary}

\newtheorem*{thm}{Theorem}

\newtheorem*{cor*}{Corollary}

\begin{document}

\title{The groups of branched twist-spun knots}

\author{Jonathan A. Hillman }
\address{School of Mathematics and Statistics\\
     University of Sydney, NSW 2006\\
      Australia }

\email{jonathanhillman47@gmail.com}

\begin{abstract}
We characterize the groups of branched twist spins of classical knots in terms of 3-manifold groups, and also give a purely algebraic,  
conjectural characterization in terms of $PD_3$-groups.
We show also that each group is the group of at most finitely many branched twist spins.
\end{abstract}

\keywords{cyclic branched cover, group, knot, twist spin}

\subjclass{57K45}

\maketitle

In this note we shall give a characterization of the groups of branched twist spins
in terms of 3-manifold groups (Theorem 2),
and then give an alternative, purely algebraic formulation,
which depends on the assumption that every $PD_3$-group
is the fundamental group of an aspherical closed 3-manifold.
The constructive aspect of the argument derives from the
work of S. P. Plotnick \cite{Pl83,Pl86}.
We show also that if two fibred knots have irreducible fibre 
and isomorphic groups then the knot manifolds are homeomorphic,
unless the fibres are lens spaces (Theorem 1), 
and that if the group of a 2-knot $K$ is isomorphic to the group
of a branched twist spin $\tau_{m,s}k$ of a prime knot $k$
then the knot manifolds are $s$-cobordant (Corollary 5).
Finally, we show that each group is the group of at most finitely many 
branched twist spins (Theorem 12).
The latter result extends the work of M. Fukuda and M. Ishikawa,
who show that if $k_1$ and $k_2$ are distinct knots and $m_1\not=m_2$ 
then  $\tau_{m_1,s_1}k_1$ and $\tau_{m_2,s_2}k_2$ are distinct (for any $s_1, s_2$)
\cite{Fu22}, and that if $k_1$ and $k_2$ are distinct prime knots
then in most cases $\tau_{m,s}k_1$ and $\tau_{m,s}k_2$ are distinct \cite {FI23}.

\section{terminology}

We review briefly some knot-theoretic terminology used here.
(See \cite[Chapter 14]{FMGK} for more details.)

An {\it$n$-knot\/} is a locally flat embedding $K:S^n\to{S^{n+2}}$.
(Our interest here is in the cases $n=1$ and 2.)
The {\it exterior\/} $X(K)$ is the closed complement of a tubular neighbourhood of $K(S^n)$ 
in $S^{n+2}$, and $\partial{X(K)}\cong{S^n}\times{S^1}$.
The {\it knot group\/} is $\pi{K}=\pi_1(X(K))$.
If we fix orientations for the spheres then the boundary of a disc $D^2$ transverse to $K$
determines a well-defined conjugacy class of {\it meridians}.
The knot group is normally generated by the image of the meridians.
The {\it knot manifold\/} $M(K)=X(K)\cup{D^{n+1}\times{S^1}}$ is 
the closed $(n+2)$-manifold obtained by elementary surgery on the knot. 
(If $n=1$ we require that $K$ have framing 0.)
If $K_1$ and $K_2$ are two $n$-knots then we shall write $K_1\cong{K_2}$ 
if there is a self-homeomorphism $h$ of $S^{n+2}$
such that $h(K_1(S^n))=K_2(S^n)$.
(We use this equivalence relation rather than ambient isotopy
because the invariant of primary interest here is the knot group,
which does not reflect the orientations.)

A {\it weight element\/} in $\pi{K}$ is an element which normally generates the group,
and a {\it weight class\/} is the conjugacy class of a weight element.
In general,  a knot group may have many weight classes.
If $n>1$ then $\pi_1(M(K))\cong\pi{K}$ and each weight class determines 
an isotopy class of simple closed curves in $M(K)$.
Elementary surgery on such a curve gives a homotopy $(n+2)$-sphere
$X(K)\cup_fS^n\times{D^2}$,
and the cocore $S^n\times\{0\}$ of the surgery is an $n$-knot.
However there are two possible framings for this surgery, 
and the corresponding knots $K$ and $K^*$ may be distinct.
We call the knot $K^*$ the ``Gluck reconstruction" of $K$,
and say that $K$ is {\it reflexive} if $K\cong{K^*}$.
Thus we may recover $K$ from $M(K)$ and the weight class of a meridian,
up to an ambiguity of order $\leq2$.

Knot groups have abelianization $\mathbb{Z}$, 
and so every knot group $\pi$ is a semidirect product $\pi'\rtimes_\theta\mathbb{Z}$,
where $\pi'$ is the commutator subgroup.
The characteristic map $\theta$ is induced by conjugation by a meridian,
and the image of $\theta$ on the outer automorphism group $Out(\nu)$ 
depends only on the weight class.
Two knot groups $\nu\rtimes_\sigma\mathbb{Z}$ and $\nu\rtimes_\theta\mathbb{Z}$ 
with isomorphic commutator subgroup $\nu$ are isomorphic if and only if the images 
of $\sigma$ and $\theta$ in $Out(\nu)$ are conjugate, 
up to inversion \cite{Pl83}.

If $K$ is a knot then the proper invariant determined by the meridians is 
 the {\it weight orbit},
which is the orbit of a meridian under the action of $Aut(\pi{K})$ \cite[14.\S7]{FMGK}.
(Different isomorphisms between the same pair of groups may not carry 
a given element to conjugate elements, 
but the images will agree up to the action of an automorphism of the target group.)
Let $\theta$ be the automorphism of $\pi'$ induced by conjugation by a weight element $t$
in $\pi$.
Then the subgroup of automorphisms which induce the identity on $\pi/\pi'$ is
\[
Aut^+(\pi{K})=\{(g,d)\in\pi'\rtimes{Aut(\pi')}\mid
d\theta{d^{-1}}\theta^{-1}=c_g\}
\] 
where $c_g(x)=gxg^{-1}$ for $x\in\nu$.
The automorphism $(g,d)$ acts on $\pi'$ through $d$ and sends $t$ to $gt$.
The orbit of a meridian under $Aut^+(\pi{K})$ is called a {\it strict\/} weight orbit.
Each strict weight orbit contains at most $|Out(\nu)|$ weight classes.
In many cases considered in this note the knot group determines the knot manifold 
up to $s$-cobordism, 
and the knot group together with the weight orbit determines the knot exterior
up to relative $s$-cobordism.

Let $M_m(k)$ be the $m$-fold cyclic branched cover of $S^3$, 
branched over a knot $k$.
If $k$ is prime and $m\geq3$ then $M_m(k)$ is generically aspherical,
with the following exceptions.
If $m=3,4 $ or 5 and $k$ is the trefoil knot $3_1$ then $\nu=\pi_1(M_m(k))\cong{Q(8)}$,
$T^*_1$ or $I^*$, respectively, 
and if $m=3$ and $k$ is the (2,5)-torus knot $5_1$  and 
then $\nu\cong{I^*}$.
If $m=3$ and $k$ is the figure eight knot $4_1$ then
$M_m(k)$ is the Hantzsche-Wendt flat 3-manifold.
Otherwise,  $M_m(k)$ is an $\widetilde{\mathbb{SL}}$-manifold if $k$ is a torus knot,
is hyperbolic if $k$ is simple but not a torus knot (and $(k,m)\not=(4_1,3)$),
and is Haken if $k$ is a satellite knot \cite{Du88}.

The 2-fold branched cover of a knot  $k$ is Seifert fibred if and only if 
$k$ is either a torus knot or a Montesinos knot.
In the latter case the covering involution acts non-trivially on the fibre \cite{Mo73}.
In particular, if $k$ is a 2-bridge knot then $M_2(k)$ is a lens space.
The other elliptic 3-manifolds (excepting those with group 
$Q(8)\times\mathbb{Z}/d\mathbb{Z}$ for some $d>1$)
are 2-fold branched covers of certain pretzel knots \cite{Mo73}.
In all cases,  $M_m(k)$ is either aspherical or has universal cover $S^3$,
and $\nu$ is either a $PD_3$-group or is finite.

All the 2-knot groups with finite commutator subgroup (excepting 
$\pi'\cong{Q(8)\times\mathbb{Z}/d\mathbb{Z}}$, $d>1$) 
are realized by 2-twist spins of 2-bridge knots or pretzel knots \cite{Yo80}.
The 2-twist spins of Montesinos knots with infinite commutator subgroup
$\pi'$ have abelian normal subgroups of rank 2,
which are not central.

Let $G'$ and $\zeta{G}$ denote the commutator subgroup and centre of a group $G$,
respectively.

\section{fibred knots}

An $n$-knot $K$ is {\it fibred\/} if there is an $(n+1)$-manifold $F$
(the fibre) with boundary $\partial{F}=S^n$ and a self-homeomorphism  
$c:F\to{F}$ (the monodromy) such that the knot exterior $X(K)$ is the mapping torus 
\[
MT(c)=F\times[0,1]/(f,1)\sim(c(f),0), ~\forall{f}\in{F}.
\]
The {\it closed fibre\/} of the knot is $\hat{F}=F\cup{D^{n+1}}$,
and the {\it closed monodromy} $\hat{c}$
is the extension of the monodromy $c$ to $\hat{F}$,
obtained by coning off the boundary.
Thus $\hat{c}$ has a fixed point, 
which determines a section $S^1\subset{MT(\hat{c})}$.
If $K$ is fibred, with fibre $F$,
then $\pi{K}'\cong\nu=\pi_1(F)$ and $\pi\cong\nu\rtimes_{c_*}\mathbb{Z}$.
Fibred 2-knots with monodromy of order 2 are determined by their complements, 
but no fibred 2-knot with monodromy of finite odd order $m>1$ is reflexive
\cite[Theorem 6.2]{Pl86}.

In the classical case $n=1$ a knot $k$ is fibred if and only if 
$\pi{k}'$ is finitely generated, and then it is a free group.
If $K$ is  a fibred 2-knot then the fibre is determined by $\pi{K}'$
(up to lens space summands).
Standard results of 3-manifold theory show that the knot group of a fibred 2-knot 
with irreducible fibre determines the knot manifold,
except when $\pi'$ is finite cyclic.
(The results invoked in Theorem 1 are scattered across a number of original papers, 
and there is no one convenient reference.) 
The case of twist spins of torus knots was treated in
 \cite[Theorem 16.6]{FMGK}.

\begin{theorem}
If $K_1$ and $K_2$ are fibred $2$-knots such that $\pi=\pi{K_1}\cong\pi{K_2}$ and 
$\pi'$ is indecomposable but not finite cyclic then $M(K_1)\cong{M(K_2)}$.
\end{theorem}

\begin{proof}
Let $N$ be the closed fibre of $K_1$. 
Then $\pi_1(N)\cong\pi'$, and so $N$ is either aspherical
or is a quotent of $S^3$.
If $N$ is aspherical it is either Haken, hyperbolic or Seifert fibred.
In each case,  the homotopy class of the monodromy of an $N$-bundle over $S^1$ 
is determined by its image in $Out(\pi')$,
and homotopic self-homeomorphisms of $N$ are isotopic.
(See \cite{Mo68} and \cite{GMT03} for hyperbolic 3-manifolds,
\cite{Wd68} and \cite{Zn86} for Haken 3-manifolds, 
\cite{Sc85} and \cite{BO91} for the infinite Seifert fibred cases,
and \cite{McC02} for the cases with $\pi'$ finite.)
\end{proof}

If $\pi'\cong*^rG_i$ where $G_i$ has one end then 
the kernel of the natural map from the group of homotopy classes of homotopy self-equivalences of $N$  to $Out(\pi')$ has order $2^{r-1}$ \cite{McC81},
and so the homotopy class of the monodromy is not determined by its image in $Out(\pi')$
if $r>1$.
Moreover, ``homotopy implies isotopy" may no longer hold \cite{FW86}.

\section{the groups of  branched twist spins}

Let $\tau_{m,s}k$ be the $s$-fold branched cyclic cover of the $m$-twist spin $\tau_mk$.
(We shall refer to such 2-knots $\tau_{m,s}k$ as ``branched $m$-twist spins", 
or just ``branched twist spins", for brevity,
and  write $\tau_mk$ instead of $\tau_{m,1}k$.)
The knot manifold $M(\tau_{m,s}k)$ is the mapping torus of the $s$th power of 
of a meridianal generator of the monodromy of $M_m(k)$ over $S^3$.
There are $\varphi(m)=|\mathbb{Z}/m\mathbb{Z}^\times|$ possible choices for $s$
{\it mod} $(m)$, and clearly $M(\tau_{m,s+m}k)\cong{M(\tau_{m,s}k)}$.
However $\tau_{m,s+m}k$ is actually $(\tau_{m,s}k)^*$ \cite{Pl84}.
(The natural framings of the normal bundles of the canonical sections of the mapping
tori differ by one full rotation.)
The knots $\tau_{m,-s}k$ and $\tau_{m,s}k$ are equivalent 
via an orientation-reversing homeomorphism.
Hence these branched covers represent at most $\varphi(m)$ distinct knots.
A simple but crucial observation for our purposes is that $\zeta\pi\tau_{m,s}k\cong\mathbb{Z}$ if and only if $k$ is not a (nontrivial) torus knot.

S.P. Plotnick has shown that (modulo the 3-dimensional Poincar\'e Conjecture)
 a fibred knot is  a branched twist spin if and only if 
the monodromy (of the fibration of the knot exterior) has finite order \cite{Pl86}.
If $m$ is odd (and $>1$) then $\tau_mk^*\not\cong\tau_mk$ \cite{Pl86}.
On the other hand,  $\tau_2k$ is always reflexive.

If $k=k_{p,q}$ is the $(p,q)$-torus knot then $M_m(k)$ is the Brieskorn 3-manifold
\[
M(m,p,q)=\{(u,v,w)\in{S^5\subset\mathbb{C}^3}\mid{u^m+v^p+w^q=0}\},
\]
and the monodromy is generated by $(u,v,w)\mapsto(\zeta_mu,v,w)$,
where $\zeta_m=exp(\frac{2\pi{i}}m)$ is a primitive $m$th root of unity.
The map from $M(m,p,q)$ to $\mathbb{CP}^1=S^2$ given by $(u,v,w)\mapsto[v^p:w^q]$
is invariant under the monodromy, 
and induces a Seifert fibration of $M(\tau_{m,s}k)$ 
over the orbifold $S^2(m,p,q)$, with general fibre the torus,
for each $s$ with $(s,m)=1$.

\begin{theorem}
A knot group $\pi$ is the group of a branched $m$-twist spin of a classical knot,
where $m>1$,
if and only if
\begin{enumerate}
\item{}there is a weight element $t$ such that $t^m$ is central;
\item{}the indecomposable factors of the commutator subgroup $\pi'$ are $3$-manifold groups;
\item{}conjugation by $t$ induces an automorphism of each such factor;
\item{}if $G_i$ is an infinite factor of $\pi'$ this automorphism has order $m$;
\item{}$\pi'$ has no $\mathbb{Z}$ factor and no factor $Q(8)\times\mathbb{Z}/d\mathbb{Z}$ with $d>1$;
\item{} if $m=3$ the only finite factors of $\pi'$ are  $Q(8)$ or $I^*$; 
if $m=4$ only $T^*_1$; if $m=5$ only $I^*$; and if $m>5$ then $\pi$ is torsion-free.
\end{enumerate}
\end{theorem}

\begin{proof}
Let $k$ be a 1-knot, and let $N=M_m(k)$ be the $m$-fold
cyclic branched cover of $S^3$, branched over $k$.
If $N$ is aspherical then the monodromy induces an automorphism of 
$\nu=\pi_1(N)$ of order $m$ \cite{CR72}. 
If $\zeta\nu=1$ then $Inn(\nu)\cong\nu$ is torsion-free, 
and so the image of the monodromy in $Out(\nu)=Aut(\nu)/Inn(\nu)$ also has order $m$.
If $k=\Sigma{k_i}$ is composite then $N$ is the connected sum of the cyclic branched covers of the summands (which are either aspherical or have finite fundamental group)
and $\tau_{m,s}k=\Sigma\tau_{m,s}k_i$ is the corresponding sum.

The necessity of the conditions follows from the above observations.
Suppose that they hold.
Condition (1) implies that the subgroup generated by $\pi'$ and $t^m$ is isomorphic to 
$\pi'\times\mathbb{Z}$, and so $\pi'$ is finitely presentable.
Hence $\pi'$ has a finite factorization into indecomposables.
Conditions (3) and (4) and the fact that $\tau_{m,s}\Sigma{k_i}=\Sigma\tau_{m,s}k_i$
allow us to reduce to the case when $\pi'$ is indecomposable.
If $\nu=\pi'$ is finite then suitable twist spins of 2-bridge knots, pretzel knots
or small torus knots give all the possible examples.
If $\pi'$ has one end and $\zeta\pi'\not=1$ then this is essentially
\cite[Theorem 16.5]{FMGK}.
The rest of the proof follows as indicated in the final paragraph of \cite[16.\S3]{FMGK}.
The key point is that the group 
$\pi/\langle{t^m}\rangle\cong\pi'\rtimes\mathbb{Z}/m\mathbb{Z}$ acts on the universal covering space $\widetilde{N}$, and the image of $t$ has 
a connected, non-empty fixed point set, by Smith theory.
Let $c$ be the self-homeomorphism of $N_o=\overline{N\setminus{D^3}}$ 
corresponding to $t$.
Then $MT(c)=X(K)$ for a 2-knot $K$, and $c$ has finite order.
Since the Poincar\'e Conjecture is now known to hold,
$X(K)$ is the exterior of a branched twist spin \cite[Proposition 6.1]{Pl86}.
\end{proof}

\begin{cor}
Let $K=\tau_{m,s}k$ and $\pi=\pi{K}$.
Then the image of a meridian for $K$ in $\pi/\zeta\pi$ has order $m$.
If $\pi'$ has at least one factor which has one end and trivial centre then $m=[\pi:\pi'.\zeta\pi]$.
\end{cor}

\begin{proof}
If $k$ is not a torus knot then both assertions follow immediately 
from the first paragraph of the proof of the theorem.
If $k=k_{p,q}$ then  $\zeta\pi{K}'\not=1$ 
and the image of a meridian of $K$ in 
$\pi{K}/\zeta\pi{K}\cong \pi^{orb}(S^2(m,p,q))$
has order $m$ \cite[Lemma 16.7]{FMGK}.
\end{proof}

Suppose that $k_{p,q}$ is a nontrivial torus knot,
and that $K=\tau_{m,s}k_{p,q}$ is also an $m_2$-twist spin $\tau_{m_2,s_2}k_2$
for some other knot $k_2$ and order of twisting $m_2$.
Since $\zeta\pi{K}$ is not cyclic, $k_2$ is also a torus knot,  
and the images of the meridians for $k_{p,q}$ and $k_2$ 
in the knot group $\pi{K}$ are conjugate.
Hence the orders of the images in $\pi{K}/\zeta\pi{K}$ are the same.
Thus the order of twisting $m$ of a branched twist-spun knot is determined 
by the knot in all cases. (This was first proven in \cite{Fu22}.)

If $d>1$ then $\nu_d=Q(8)\times\mathbb{Z}/d\mathbb{Z}$ has an automorphism $\theta$
of order 6 such that $\nu_d\rtimes_\theta\mathbb{Z}$ is the group of a fibred 2-knot,
but no self-homeomorphism of $S^3/\nu_d$ of order 6 has a fixed point.
Hence $\nu_d$ is not the commutator subgroup of a twist spun knot.
There are also fibred 2-knots with knot manifolds whose groups have centre $\mathbb{Z}^2$ but which do not satisfy condition (1) \cite[page 318]{FMGK}.
Thus this condition is not a consequence of the outer automorphism having finite order.

Let $N$ be a homology 3-sphere obtained by Dehn surgery on a knot $k$.
Then $\nu=\pi_1(N)$ has weight one, and so does $\pi=\nu\times\mathbb{Z}$.
Surgery on a loop in $N\times{S^1}$ representing a weight class
gives a homotopy 4-sphere, and the cocore of the surgery 
is a 2-knot with group $\nu\times\mathbb{Z}$.
If $k$ is a simple non-torus knot (such as $4_1$) then for all but finitely many Dehn 
surgeries $N$ is hyperbolic.
The manifold $M=N\times{S^1}$ is then a 2-knot manifold which is a 
$\mathbb{H}^3\times\mathbb{E}^1$-manifold.
However $\pi$ is not isomorphic to $\pi\tau_{m,s}k$ for any knot $k$,
since $[\pi:\pi'.\zeta\pi]=1$, and 1-twist spins are trivial.
(This answers the question following \cite[Corollary 17.11.1]{FMGK}.)

In such cases the outer automorphism determined by the monodromy has finite order, 
but does not lift to an automorphism of finite order.
Thus we cannot use Smith theory to show that the monodromy of the knot exterior 
has finite order, which is an essential hypotheses of \cite[Proposition 6.1]{Pl86}.

Condition (2) of Theorem 2 is unsatisfactory as part of an algebraic characterization of such groups.
If $K$ is a 2-knot such that $\pi'=\pi{K}'$ is finitely presentable
then $M(K)'$ is a $PD_3$-complex \cite[Theorem 4.5]{FMGK}.
Hence the indecomposable factors of $\pi'$ are either $PD_3$-groups or are virtually free.
Thus if (1) holds (so that $\pi'$ is finitely presentable) and if every $PD_3$-group is
the fundamental group of an aspherical closed $3$-manifold then we may then replace 
(2) by the condition

\smallskip
{\it (2')
$\pi'$ has no indecomposable factor which is an infinite virtually free group other than $\mathbb{Z}$}.

\smallskip
\noindent{}In particular,
if $\pi'$ is indecomposable and infinite  conditions (2)--(6) in Theorem 2 could 
be replaced by

{\it (2'') $\pi$ is a $PD_4$-group and conjugation by a weight element $t$ 
induces an automorphism of $\pi'$ of order $m$}.

\begin{theorem}
The knot manifold $M(\tau_{m,s}k)$ has a geometric decomposition if and only if $k$ is 
a prime knot.
\end{theorem}

\begin{proof}
A closed orientable 4-manifold $M$ with a geometric decomposition and $\chi(M)=0$ is 
either a $\mathbb{S}^2\times\mathbb{E}^2$- or $\mathbb{S}^3\times\mathbb{E}^1$-manifold,
or is aspherical \cite[Theorem 7.2]{FMGK}.

If $k$ is prime and $M_m(k)$ is covered by $S^3$ then $M(\tau_{m,s}k)$
is a $\mathbb{S}^3\times\mathbb{E}^1$-manifold.
If $M_m(k)$ is aspherical then it has a JSJ decomposition which is equivariant 
with respect to the action of the monodromy. 
Since the (closed) monodromy has finite order the pieces of the decomposition of $M_m(k)$
give rise to a decomposition of $M(\tau_{m,s}k)$ with the corresponding product geometry
$\mathbb{H}^3\times\mathbb{E}^1$, $\mathbb{H}^2\times\mathbb{E}^2$
or $\widetilde{\mathbb{SL}}\times\mathbb{E}^1$.

If $k$ is composite then $\pi'$ is a proper free product and $\pi_2(M_m(k))\not=0$.
Hence $\pi\tau_{m,s}k$ is not virtually abelian and $M(\tau_{m,s}k)$ is not aspherical,
so $M(\tau_{m,s}k)$ has no geometric decomposition.
\end{proof}

It can be shown that finite volume $\mathbb{H}^2\times\mathbb{E}^1$-manifolds
which are not closed are also $\widetilde{\mathbb{SL}}$-manifolds,
and so proper pieces of type $\mathbb{H}^2\times\mathbb{E}^2$ are also of type 
$\widetilde{\mathbb{SL}}\times\mathbb{E}^1$.

\begin{cor}
If $K$ is a $2$-knot with $\pi{K}\cong\pi\tau_{m,s}k$, 
where $k$ is a prime knot and $M_m(k)$ is aspherical,
then $M(K)$ is $s$-cobordant to $M(\tau_{m,s}k)$.
\end{cor}

\begin{proof}
This follows from \cite[Theorem 9.12]{FMGK}
and \cite[Theorem 3.4]{Hi12}, 
since $Wh(\pi\times\mathbb{Z})=0$,
as explained in \cite[9.\S6]{FMGK}.
\end{proof}

\section{uniqueness}

In this section we shall summarize the results of Fukuda and Ishikawa, 
and complete their treatment of the torus knot case.
We shall also show that a 2-knot group can be the group of only finitely 
many branched twist spins.

We begin with a simple argument for the torus-knot case.

\begin{theorem}
Let $\pi$ be a $2$-knot group such that $\zeta\pi\cong\mathbb{Z}^2$.
Then $\pi$ is the group of at most finitely many branched twist spins.
\end{theorem}

\begin{proof}
If $\pi=\pi\tau_{m,s}k$ then $k$ is a torus knot $k_{p,q}$,
and $M(\tau_{m,s}k)$ is Seifert fibred over $S^2(m,p,q)$.
Hence $\pi/\zeta\pi\cong\pi^{orb}(S^2(m,p,q))$, and so
$\pi$ determines the triple $\{m,p,q\}$.
Thus there are at most 3 possibilities for $m$.
For each value of $m$ there are at most $\varphi(m)$ distinct branched twist spins.
\end{proof}

\begin{cor}
\label{torus knot}
The torus knot $k_{p,q}$ is determined by $\pi$ and $m$.
\qed
\end{cor}

Let $m,p,q>1$ be pairwise relatively prime integers.
Then $M(m,p,q)$ is a homology 3-sphere, 
and  $M_m(k_{p,q})\cong{M_p(k_{m,q})}\cong{M_q(k_{m,p})}\cong{M(m,p,q)}$.
If $\nu=\pi_1(M(m,p,q))$ the  the corresponding twist spins all have 
$\pi\cong\nu\times\mathbb{Z}$.
Hence $\pi=\pi'.\zeta\pi$, and $m>[\pi:\pi'.\zeta\pi]=1$.
Thus the order of the twisting $m$ is not determined by the group.
However, the order of the twisting must be one of $m,p$ or $q$.
(Similarly, if $(m,q)=(p,q)=1$ but $(m,p)>1$ then the order of the twisting is $m$ or $q$.
For example, $\pi\tau_4k_{2,3}\cong\pi\tau_2k_{3,4}\cong{T^*_1}\rtimes\mathbb{Z}$.)

Fukuda and Ishikawa give a stronger result for branched twist spins
of the other prime knots.
(They also use knot determinants to prove Corollary \ref{torus knot} above,
with some restrictions on the parameters $p,q$.)
A knot $k$ is {\it Conway reducible}  if it admits an essential Conway sphere.

\begin{thm}
[Fukuda--Ishikawa \cite{FI23}]
Let $k_1$ and $k_2$ be two prime knots, and suppose that 
$\pi\tau_{m,s}k_1\cong\pi\tau_{m,s}k_2$ for some $m,s$ with $(m,s)=1$.
Suppose also that either  $m=2$ and $k_1$ and $k_2$ are each 
satellite knots or Conway reducible
or  $m\geq3$ and $k_1$ and $k_2$ are not torus knots. 
Then $k_1\cong{k_2}$,
and so $\tau_{m,s}k_1\cong\tau_{m,s}k_2$.
\qed
\end{thm}

The key points are that $k_1$ and $k_2$ are either hyperbolic or satellites, 
and the quotient $\pi/\zeta\pi$ is the orbifold fundamental group of 3-orbifolds 
corresponding to the branched coverings of $(S^3,k_1)$ and $(S^3,k_2)$.
These orbifolds are hyperbolic or Haken, 
and therefore are determined by their fundamental groups.
The argument of \cite{FI23} shows in fact that if  $(m,s_1)=(m,s_2)=1$ and
$\pi\tau_{m,s_1}k_1\cong\pi\tau_{m,s_2}k_2$ then $k_1\cong{k_2}$. 
However more may be needed to conclude that $\tau_{m,s_1}k_1\cong\tau_{m,s_2}k_2$.

The next result is closely related to \cite[Theorem 2.1]{Pl83}.

\begin{theorem} 
Let $k$ be a knot and let $\theta$ be the automorphism of $\nu=\pi_1(M_m(k))$ 
induced by the canonical generator of the branched covering,
for some $m>1$.
Let $s_1,s_2$ be integers such that $(m,s_1)=(m,s_2)=1$.
If $\tau_{m,s_1}k\cong\tau_{m,s_2}k$ then $\theta^{s_1}$ is conjugate in $Aut(\nu)$
to $\theta^{s_2}$ or $\theta^{-s_2}$.
Conversely, if $\theta^{s_1}$ is conjugate in $Aut(\nu)$
to $\theta^{s_2}$ or $\theta^{-s_2}$
then $\pi\tau_{m,s_2}k\cong\pi\tau_{m,s_1}k$ and the meridians 
of $\tau_{m,s_1}k$ and $\tau_{m,s_2}k$ are in the same weight orbit,
up to inversion.
\end{theorem}

\begin{proof}
The automorphism of $\nu$ induced by the canonical meridian 
for $\tau_{m,s}k$ is $\theta^s$.
Let $t_1$ and $t_2$ be the canonical weight elements for $\tau_{m,s_1}k$ and 
$\tau_{m,s_2}k$.
A homeomorphism of $S^4$ which carries $\tau_{m,s_1}k$ to
$\tau_{m,s_2}k$ and preserves the homology class of the meridians 
induces an isomorphism $f_*:\pi\tau_{m,s_1}k\to\pi\tau_{m,s_2}k$ 
such that $f_*(t_1)=gt_2g^{-1}$, 
for some $g\in\nu$,
and $f_*(\theta^{s_1}(x))=g\theta^{s_2}(f_*(x))g^{-1}$, for all $x\in\nu$.
Hence $d=c_g^{-1}f_*|_\nu$ is an automorphism of $\nu$ 
such that $d\theta^{s_1}d^{-1}=\theta^{s_2}$.

Conversely, 
if $d\theta^{s_1}d^{-1}=\theta^{s_2}$ for some automorphism $d$ of $\nu$ then
setting $D(x)=d(x)$ for $x\in\nu$ and $d(t_1)=t_2$
defines an isomorphism $D:\pi\tau_{m,s_1}k_1\to\pi\tau_{m,s_2}k_2$
which carries $t_1$ to $t_2$. 
Thus $D$ carries the weight orbit of $t_1$ onto the weight orbit of $t_2$.
(Note that since $\theta^{s_1}$ and $\theta^{s_2}$ obviously commute,
the criterion  of \S1 for comparison of the weight orbits reduces to
requiring that $\theta^{s_1}$ be conjugate to $\theta^{s_2}$ in $Aut(\pi')$.)

The above arguments are easily adapted to the cases when we allow the
homeomorphism carrying the image of one knot onto the other to reverse the homology class of the meridians.
\end{proof}

We know from Theorem 1 that if $k$ is prime then under the assumptions
of the theorem $M(\tau_{m,s_1}k)\cong{M(\tau_{m,s_2}k)}$.
There is a stronger result when $k$ is a torus knot and $\nu$ is not virtually solvable.

\begin{cor}
If $k=k_{p,q}$ where $\frac1m+\frac1p+\frac1q<1$
then $X(\tau_{m,s_1}k)\cong{X(\tau_{m,s_2}k)}$ if and only if
$\theta^{s_1}$ is conjugate in $Aut(\nu)$
to $\theta^{s_2}$ or $\theta^{-s_2}$.
\end{cor}

\begin{proof}
A homeomorphism of knot exteriors must carry meridians to meridians 
(up to inversion),  and so the condition is necessary.

The hypothesis $\frac1m+\frac1p+\frac1q<1$ is equivalent to $M_m(k)$
being aspherical and $\nu$ not solvable.
Hence every automorphism of $\pi\tau_{m,s_2}k$ is realized by a self-homeomorphism
of $M(\tau_{m,s_2}k)$ \cite[Theorem 11.2.4]{LR}.
Thus we may assume that there is a homeomorphism $h$ realizing 
an isomorphism $D:\pi\tau_{m,s_1}k_1\to\pi\tau_{m,s_2}k_2$ 
which carries the (free) isotopy class of a meridian for  $\tau_{m,s_1}k$
in $M(\tau_{m,s_1}k)$ to a meridian for $\tau_{m,s_2}k$ in  $M(\tau_{m,s_2}k)$.
Such a homeomorphism clearly restricts to a homeomorphism of the exteriors.
The corollary then follows from the theorem.
\end{proof}

Note also that since torus knots are invertible, 
their (branched) twist spins are $+$amphicheiral \cite{Li85}.
Thus we may assume that $h$ is orientation-preserving.

TOP surgery may be used to give a similar result for $\tau_6k_{2,3}$
(corresponding to $\frac1m+\frac1p+\frac1q=1$),
and more generally when $M(K)$ is aspherical and $\pi$ is torsion-free and solvable.
This is so if $K=\tau_34_1$ or $K=\tau_2k$ for certain Montesinos knots $k$.
(See \cite[16.\S4]{FMGK}.)

A further argument comes close to determining the number of distinct branched 
twist spins of torus knots.

\begin{theorem}
Let $m,p,q, s,t$ be positive integers such that $(p,q)=(m,s)=(m,t)=1$ and
$\frac1m+\frac1p+\frac1q<1$.
If $\tau_{m,s}k_{p,q}\cong\tau_{m,t}k_{p,q}$ then $s^4\equiv{t^4}$ mod $(m)$.
\end{theorem}

\begin{proof}
We may assume that $m=5$ or $m\geq7$, for otherwise $s^4=t^4=1$.
Let $\nu=\pi_1(M(m,p,q))$.
The base $B$ of the Seifert fibration of $M(m,p,q)$ is an 
orientable hyperbolic 2-orbifold, and $\pi^{orb}B=\nu/\zeta\nu$.
The group of lifts of orientation preserving automorphisms of $B$ to 
$\widetilde{B}=\mathbb{H}^2$ is the normalizer of $\nu$ in $PSL(2,\mathbb{R})$.
The characteristic automorphism $\theta$ corresponding to meridians 
of $\tau_mk_{p,q}$ acts orientably on $B$, and so determines an element of 
$PSL(2,\mathbb{R})$.

If $\tau_{m,s}k_{p,q}\cong\tau_{m,t}k_{p,q}$ then 
$\theta^t=\alpha\theta^{\pm{s}}\alpha^{-1}$ in $Aut(\nu)$.
Hence $\theta^{t^2}=\alpha^2\theta^{s^2}\alpha^{-2}$, 
and $\alpha^2$ is orientation preserving.
Let $\bar{g}$ and $\bar{h}$ be the images of $\theta$ 
and $\alpha^2$ in $PSL(2,\mathbb{R})$.
Then $\bar{g}$ has order $m$, and $\bar{h}\bar{g}^{s^2}\bar{h}^{-1}=\bar{g}^{t^2}$.
Since  $\langle\bar{g},\bar{h}\rangle$ 
is a subgroup of $PSL(2,\mathbb{R})$ and $m>2$,
this subgroup is either abelian or dihedral.
Hence $s^2\equiv\pm{t^2}$ {\it mod} $(m)$, 
and so $s^4\equiv{t^4}$  {\it mod} $(m)$.
\end{proof}

If $m>10$ then $\varphi(m)>4$,
and so there is an $r$ such that $(m,r)=1$ and $r^4\not\equiv1$ {\it mod\/} $(m)$.
Thus if  $(m,t)=1$ then $X(\tau_{m,rt}k_{p,q})\not\cong{X(\tau_{m,t}k_{p,q})}$, 
although the knot groups are isomorphic.

\begin{cor}
If $m=r^k$ or $2r^k$ for some prime $r\equiv3$ mod $(4)$ then
$X(\tau_{m,s}k_{p,q})\cong{X(\tau_{m,t}k_{p,q})}$ 
if and only if $s^2\equiv{t^2}$ mod $(m)$.
\end{cor}

\begin{proof}
In this case $s^4\equiv{t^4}$ {\it mod} $(m)$ if and only if $s\equiv\pm{t}$ {\it mod} $(m)$.
\end{proof}

In particular, if $m=7$ then there are $3=\frac12\varphi(7)$ distinct knot exteriors
$X(\tau_{7,s}k_{p,q})$, for each torus knot $k_{p,q}$.

If $k$ is composite then $M_m(k)$ is a proper connected sum,
whereas if $k$ is prime then $M_m(k)$ is irreducible.
Thus a branched twist spin of a prime knot is never 
a branched twist spin of a composite knot.

\begin{theorem}
If $\pi$ is a group then $\pi\cong\pi\tau_{m,s}k$ for at most finitely
many branched twist spins $\tau_{m,s}k$.
\end{theorem}

\begin{proof}
Suppose that $\pi\cong\pi\tau_{m,s}k$ for some 1-knot $k$ and some $m,s$.
Let $k=\Sigma_{i\leq{r}}{k_i}$ be the factorization of $k$ as a sum of prime knots.
Then $\pi\cong{G_1}*_\mathbb{Z}\dots*_\mathbb{Z}G_r$ 
is the free product of the $G_i=\pi\tau_{m,s}k_i$, 
with amalgamation over subgroups generated by meridians,
and $\pi'\cong\star_{i\leq{r}}{G_i'}$.
Since $\pi'$ is finitely generated, 
it is a free product of indecomposable subgroups,
and the factors which are not infinite cyclic are unique up to
conjugacy and re-indexing.
The subgroups $G_i'$ are indecomposable and none are infinite cyclic.
Hence the set $\{G_i'\}$ and the number $r$ of prime factors of $k$ 
are determined by $\pi$.

The quotient of $\pi$ by the normal closure of the subgroups $\{G_j':j\not=i\}$
is $G_i$. 
Then $G_i$ is the group of a branched twist simple knot in at most finitely many ways.
This follows from \cite{Mo73} if $G_i'$ is finite.
If $G_i'$ is infinite then $G_i'\cong\pi_1(N_i)$, 
where $N_i$ is aspherical and is Haken,  hyperbolic or Seifert fibred.
If $\zeta{G_i'}=1$ then $m=[\pi:\pi'.\zeta\pi]$, by Corollary 3,
and $k_i$ is not a torus knot.
If $m>2$ or if $m=2$ and $N_i$ is Haken
then $k_i$ is determined by $G_i$ and the pair $(m,s)$,
by the work of Fukuda and Ishikawa \cite{FI23} (cited above).
If $m=2$ and $N_i$ is hyperbolic then there are at most 9 possibilities for $k_i$ \cite{Re00}.
If $\zeta{G_i}\cong\mathbb{Z}^2$ then $k_i$ is a torus knot 
and there are at most 3 possibilities for $k_i$, by Theorem 6.
Finally, if $\zeta{G_i'}\not=1$ but $\zeta{G_i}\cong\mathbb{Z}$
then $m=2$ and $k_i$ is a Montesinos knot;
the number of possibilities for $k_i$ depends on $G_i'$ and may be large, 
but is finite.

Thus there are at most finitely many candidates for the summands $k_i$, 
and hence for $k$.
Since there are only finitely many possibilities for $m$ (and hence $s$), 
by Corollary 3 and Theorem 6, the result follows.
\end{proof}

The following corollary follows immediately from Theorems 1 and 9.

\begin{cor}
If $K$ is a $2$-knot and $\pi{K}'$ is indecomposable then 
$M(K)\\ \cong{M(\tau_{m,s}k)}$ 
for at most finitely many branched twist spins $\tau_{m,s}k$. 
\qed
\end{cor}




\begin{thebibliography}{99}

\bibitem{BO91} Boileau, M. and Otal, J.-P. Scindements de Heegaard et groupe 
des hom\'eotopies des petites vari\'et\'es de Seifert,
Invent. Math. 106 (1991), 85--107.

\bibitem{CR72} Conner, P. E. and Raymond, F. R. Manifolds with few periodic automorphisms, 
in \textit{Proceedings of the Second Conference on Compact Transformation Groups,
U. Mass., Amherst (1971)}, Part II,
Lecture Notes in Mathematics 299, Springer-Verlag,
Berlin -- Heidelberg -- New York (1972), 1-75.

\bibitem{Du88} Dunbar, W. Geometric orbifolds,

Rev.  Mat.  Complut.  1 (1988), 67--99.

\bibitem{FW86} Friedman, J. L. and Witt,  D. M.  Homotopy is not isotopy 
for homeomorphisms of 3-manifolds,
Topology 25 (1986), 35--44.

\bibitem{Fu22} Fukuda, M. 
Representations of branched twist spins with a non-trivial center of order 2,
Top. App. 365 (2021), article ID 109284, 10p. 

\bibitem{FI23} Fukuda, M. and Ishikawa, M.  
Distinguishing 2-knots admitting circle actions by fundamental groups,
Rev.  Mat.  Complut.  38 (2025), 307--317.


\bibitem{GMT03} Gabai, D., Meyerhoff, G. R. and Thurston, N. Homotopy hyperbolic 3-manifolds are hyperbolic,
Ann. Math. 157 (2003),  335--431.

\bibitem{FMGK} Hillman, J.  A. \textit{Four-Manifolds, Geometries and Knots},

GT Monograph 5, 
Geometry and Topology Publications, Warwick (2002).

2022 revision -- arXiv: math.0212142.

\bibitem{Hi12} Hillman, J. A. The groups of fibred knots,

in \textit{Geometry and Topology Down Under}, CONM 597,

American Mathematical Society, Providence R.I.  (2013), 281--294.

\bibitem{LR} Lee, K. B. and Raymond, F. \textit{Seifert Fiberings},

Math. Surveys and Monographs 166, 

American Mathematical Society, Providence R.I. (2010).

\bibitem{Li85} Litherland, R. A.  Symmetries of twist-spun knots,

in \textit{Knot Theory and Manifolds} (edited by D. Rolfsen),
Lecture Notes in Mathematics 1144, 
Springer-Verlag,
Berlin -- Heidelberg -- New York (1985),  97--107.

\bibitem{McC81} McCullough, D.  Connected sums of aspherical manifolds, 

Indiana U.  Math.  J.  30 (1981),  17--28.

\bibitem{McC02} McCullough, D.  Isometries of elliptic 3-manifolds,

J. London Math. Soc.  65 (2002), 167--182.

\bibitem{Mo73} Montesinos, J. M. Variedades de Seifert que son recubricadores ciclicos ramificados de dos hojas,  
Bol.  Soc. Mat. Mexicana 18 (1973), 1-32.

\bibitem{Mo68} Mostow, G.D. Quasiconformal mappings in $n$-space and the rigidity of hyperbolic space forms,
Publ. Math. I.H. E. S.  34 (1968),  53--104.

\bibitem{Pl83} Plotnick, S. P. The homotopy type of four-dimensional knot complements,

Math. Z. 183 (1983), 447--471.

\bibitem{Pl84} Plotnick, S.P.   Fibered knots in $S^4$ -- twisting, spinning, rolloing, surgery and branching,
in \textit{Four-Manifold Theory}, CONM 35,

American Mathematical Society, Providence R.I. (1984),  437--459.

\bibitem{Pl86} Plotnick, S. P.  Equivariant intersection forms, knots in $S^4$, 
and rotations in 2-spheres,
Trans. Amer. Math. Soc. 296 (1986), 543--575.

\bibitem{Re00} Reni, M.  On $\pi$-hyperbolic manifolds with the same 2-fold branched coverings,

Math. Ann. 316 (2000),  681--697.

\bibitem{Sc85} Scott, G.P.  Homotopy implies isotopy for some Seifert fibre spaces,

Topology 24 (1985),  341--351.

\bibitem{Wd68} Waldhausen, F.  On irreducible 3-manifolds which are sufficiently large,

Ann. Math.  87 (1968),  56--88.

\bibitem{Yo80} Yoshikawa, K. On 2-knot groups with the finite commutator subgroups,

Math. Sem. Notes 8 (1980),  321--330.

\bibitem{Zn86}  Zimmermann, B. Finite group actions on Haken 3-manifolds,

Quart. J. Math. Oxford 37 (1986),  499--511.

\end{thebibliography}
\end{document}